\documentclass[10pt]{article}

\usepackage[english]{babel}
\usepackage[utf8]{inputenc}
\usepackage[T1]{fontenc}
\usepackage[toc,page]{appendix}
\usepackage{graphicx}
\usepackage{geometry}
\usepackage{bbm}
\usepackage{dsfont}
\usepackage{alltt}
\usepackage{verbatimbox}
\usepackage{amsthm,amsfonts}
\usepackage{amsmath}
\usepackage{amssymb}
\usepackage{mathabx}
\usepackage{accents}
\usepackage{titlesec}
\usepackage{mathtools}
\usepackage{empheq}
\usepackage{tikz}
\usepackage{enumitem}

\mathtoolsset{showonlyrefs=true}

\newcommand{\real}{\mathbb{R}}

\newcommand{\x}{\xi}

\newcommand{\jap}[1]{\left\langle #1 \right\rangle}
\newcommand{\subscript}[2]{$#1 _ #2$}

\numberwithin{equation}{section}
\newcommand{\be}{\begin{equation}}
\newcommand{\ee}{\end{equation}}

\DeclareMathOperator{\sgn}{\mathrm{sgn}}

\theoremstyle{plain}
\newtheorem{thm}{Theorem}
\newtheorem*{thm*}{Theorem}

\newtheorem{assumpt}{Assumption}
\newtheorem{lem}{Lemma}

\theoremstyle{definition}

\theoremstyle{remark}
\newtheorem{nb}{Remark}

\makeatletter
\def\blfootnote{\xdef\@thefnmark{}\@footnotetext}
\makeatother

\title{Nonlinear smoothing and unconditional uniqueness for the Benjamin-Ono equation in weighted Sobolev spaces}

\date{}
\author{Simão Correia}

\begin{document}
\maketitle

\begin{abstract}
	We consider the Benjamin-Ono equation on the real line for initial data in weighted Sobolev spaces. After the application of the gauge transform, the flow is shown to be Lipschitz continuous and to present a nonlinear smoothing effect. As a consequence, unconditional uniqueness for the Benjamin-Ono equation is proved. 
\end{abstract}

\section{Introduction}
In this paper, we consider the Benjamin-Ono equation on the real line,
\begin{equation}\label{eq:bo}
\partial_tu + H\partial_{x}^2 u = \partial_x(u^2), \quad (t,x)\in [0,T]\times\real,
\end{equation}
where $H$ denotes the Hilbert transform, defined through the Fourier transform as
$$
\widehat{Hf}(\xi)=-i\sgn(\xi)\hat{f}(\xi), \quad \xi\in \real.
$$
The Benjamin-Ono equation models the propagation of unidirectional deep water waves \cite{benjamin, ono}. As a nonlinear dispersive equation, it is nonlocal, completely integrable and the nonlinear term presents a loss of derivative. In fact, since the linear dispersion effects are quite weak, one cannot handle the nonlinearity perturbatively in order to prove local well-posedness results for initial data $u_0\in H^s(\real)$. This was proven rigorously by Molinet, Saut and Tzvetkov \cite{molinetsauttzevtkov} and improved by Koch and Tzvetkov \cite{kochtzevtkov1}: the flow map is shown not to be $C^2$ for $s\in \real$, or even uniformly continuous for $s>0$.

A first step in the local well-posedness theory was given by Iorio \cite{iorio} for $s>3/2$. Several refinements ensued: Ponce \cite{ponce} for $s=3/2$, Koch-Tzevtkov \cite{kochtzevtkov2} for $s>5/4$ and Kenig-Koenig \cite{kenigkoenig} for $s>9/8$. Later, using a variant of the Hopf-Cole transform (connected to the Burgers equation), 
$$
w\sim\frac{1}{2i}ue^{-i\partial_{x}^{-1}u},
$$
Tao \cite{tao} noticed that the worst interactions in the nonlinearity disappear, allowing for a rather direct proof of  local well-posedness in $H^1(\real)$. The method was later improved by Burq-Planchon \cite{burqplanchon} for $s>1/4$ and by Ionescu-Kenig \cite{ionescukenig} for $s\ge 0$ (later revisited by Molinet-Pilod \cite{molinetpilod} and by Ifrim-Tataru \cite{ifrimtataru}). One of the main difficulties is the transfer of bounds from $u$ to $w$ (and vice-versa), which has been handled either through paralinearization or by decomposing the solution into low and high frequencies. It is worth mentioning that in \cite{ifrimtataru}, the authors prove the local well-posedness in the weighted space $L^2((1+x^2)dx)$ (see also \cite{fonsecalinaresponce}). Finally, in \cite{molinetpilod}, the solution is shown to be unique for $s>1/4$, while conditional uniqueness (that is, under the assumption that the solution belongs to some auxiliary space) holds in the class $L^\infty((0,T), H^s(\real))\cap L^4((0,T), W^{s,4}(\real))$ for $s>0$. In the periodic case, unconditional uniqueness has been proven in \cite{kishimoto1} for $s>1/6$. 

The main problem we wish to study in the context of the Benjamin-Ono equation is the nonlinear smoothing phenomena: the difference between the nonlinear and the free evolutions starting from the same initial data is in fact smoother than the initial data. For generic dispersive equations of the form $u_t+iL(D)u=N(u)$, one aims to derive the estimate
$$
\| u(t) - e^{-itL(D)}u_0\|_{L^\infty(H^{s+\epsilon}(\real))}\le C(t, \|u\|_{L^\infty((0,t), H^s(\real))}).
$$
This feature has been initially discovered by Bona and Saut \cite{BS1} in the context of generalized KdV equations and extended to many other contexts (see \cite{babin, bourgain, tzirakis6, tzirakis3, keraani, linaresscialom}, among others), by using either maximal estimates, Bourgain spaces or integration by parts in the time variable. Recently, the author and Silva \cite{CorreiaSilva} derived a unifying strategy, based on the infinite normal form reduction (INFR), to prove this phenomena for general dispersive equations and applied it to several classical examples, such as the Korteweg-de Vries and the modified Zakharov-Kuznetsov equations. In the context of the Benjamin-Ono equation on the real line, there are no results concerning this property (see \cite{HKO} for a  smoothing effect outside of the origin under high regularity and decay assumptions).

The INFR has been introduced in \cite{guo, ko, koy}. The main idea is to consider the profile of a solution $u$, $\tilde{u}(t)=e^{-itL(D)}u(t)$, whose equation concentrates all the dispersive information in an oscillatory integral. By integrating by parts in time (an idea used in \cite{babin} and \cite{tzirakis3}), one expects to gain some regularity (coming from the phase function) at the expense of having time derivatives in the nonlinearity, which, upon replacement, yield higher-order terms. Afterwards, one may integrate by parts once again, for as many times as one so desires. The INFR is simply the formal construction of an infinite iteration of this argument. 

Evidently, if the phase is stationary in time, the argument is not as direct, requiring a decoupling into resonant and nonresonant parts. This is reminiscent of the space-time resonances methodology presented in \cite{germain, masmoudi}. The main difference in their argument is the use of vector-fields to deal with time resonances, while the INFR uses the resonance condition as a restriction on the domain of the oscillatory integral. However, as one may observe in \cite{CorreiaSilva} and in Section 4, the INFR analysis ultimately hinges on the study of space-time resonances. One may even argue that the INFR is both a refinement and a simplification of the method of \cite{germain, masmoudi}. 

The INFR method has also been used by Kishimoto in \cite{kishimoto1, kishimoto2} (see also \cite{gebalin}) to prove unconditional uniqueness at low regularity. Indeed, the advantage of working with an infinite expansion in arbitrarily high-order terms is that the analysis can be carried out without the use of any auxiliary space. Moreover, it is worth pointing out that in all nonlinear smoothing results for one-dimensional dispersive equations presented in \cite{CorreiaSilva} and in the present paper, the critical regularity for local well-posedness is reached. We believe this fact is not random and further research into this matter is necessary.

As expected, we are unable to prove nonlinear smoothing directly for \eqref{eq:bo}: a consequence of our methodology is a Lipschitz continuity estimate for the flow, which would contradict \cite{kochtzevtkov1}. We proceed as for the derivative Schrödinger equation (see \cite{CorreiaSilva, tzirakis6}) and prove the phenomenon for the gauged version of \eqref{eq:bo}, inspired in  \cite{kishimoto1} and \cite{tao}. Through the INFR, we are able to prove nonlinear smoothing and unconditional uniqueness for the Benjamin-Ono equation in weighted Sobolev spaces.
\begin{thm}\label{teo:principal}
	Fix $s>0$. Define
	$$
	\Sigma=\left\{u \in H^s(\real): xu\in L^2,\ \hat{u}(0)=0 \right\}.
	$$
	\begin{enumerate}
		\item 	Given $u_0\in \Sigma$, there exists $T>0$ and a unique local solution of \eqref{eq:bo}
		$$
		u\in C([0,T], H^s(\real))\cap L^\infty((0,T), \Sigma)
		$$ with $u(0)=u_0$.
		\item The gauge mapping $\mathcal{G}: \Sigma \to  H^{s+1}(\real)$
		$$
		u \mapsto \mathcal{G}u=e^{-i\partial_x^{-1}u/2}-1
		$$
		is well-defined and it is continuous.
		\item Given solutions $u, v\in C([0,T], \Sigma)$ of \eqref{eq:bo}, one has the local Lipschitz estimate
		$$
		\|\mathcal{G}(u) - \mathcal{G}(v)\|_{L^\infty((0,T), H^{s+1})}\lesssim C\left(T,\|\mathcal{G}(u(0))\|_{H^{s+1}}, \|\mathcal{G}(v(0))\|_{ H^{s+1}}\right)\|\mathcal{G}(u(0)) - \mathcal{G}(v(0))\|_{ H^{s+1}}
		$$
		Moreover, for $\epsilon<\max\{ s, 3/4 \}$, 
		$$
		\|\mathcal{G}(u) - e^{-tH\partial_{x}^2}\mathcal{G}(u(0))\|_{L^\infty((0,T), H^{s+\epsilon+1})}\lesssim C(T, \|\mathcal{G}(u)\|_{L^\infty((0,T),H^{s+1})}).
		$$
	\end{enumerate}

\end{thm}

The use of a weighted space is necessary in order for the gauge mapping to lie in $H^{s+1}(\real)$. This is a crucial part of our analysis. To lift this restriction, one must find a way to treat the low-frequency terms, possibly as in \cite{ionescukenig}. On the other hand, the necessity of weighted spaces to prove nonlinear smoothing is not surprising, as it has been observed for the KdV equation in \cite{imt}. 

The remainder of this work is organized as follows: in Section 2, we recall the gauge transformation and derive the gauged equation. In Section 3, we give a brief overview of the INFR method and point out the required adaptations for the Benjamin-Ono equation. Finally, in Section 4, we prove the two essential estimates that allow for the application of the INFR, thus concluding the proof of the main theorem. 
%

\section{The gauge transformation}

Given $v\in C([0,T], \mathcal{S}'(\real))$, we define 
$$
\tilde{v}(t):=\widehat{e^{tH\partial_{x}^2}v(t)}.
$$
Given an initial data $u_0\in \Sigma$, \cite[Proposition 7.2]{ifrimtataru} ensures that any solution $u\in C([0,T], H^s(\real))$ satisfies
$$
\|\partial_\xi \tilde{u}(t)\|_{L^2} =\|(x-2tH\partial_x)u(t)\|_{L^2}\lesssim C(t, \|u_0\|_{L^2}, \|xu_0\|_{L^2}).
$$
Moreover, taking into account that
$$
\tilde{u}_t(t,\xi) = i\xi e^{it|\xi|\xi} \widehat{u^2}(t,\xi),
$$
the application of Riemann-Lebesgue lemma shows that $\tilde{u}_t \in C([0,T]\times \real)$. Consequently, $\hat{u}(t,0)= \tilde{u}(t,0)=0$ for $t\in [0,T]$. These properties allow us to define the antiderivative of $u$,
$$
\hat{F}(t,\xi)=\frac{\hat{u}(t,\xi)}{i\xi}, \quad (t,\xi)\in [0,T]\times (\real\setminus\{0\}).
$$
By Hardy's inequality, $F\in L^\infty([0,T], H^{s+1}(\real))$. From \eqref{eq:bo}, we deduce
\begin{equation}\label{eq:eqF}
	F_t + H  F_{xx}= \frac{1}{2}(F_x)^2.
\end{equation}
Take $\phi=\mathbbm{1}_{[-1,1]}$ and define the projections
$$
\widehat{P_\pm f} = \mathbbm{1}_{\pm \xi>0}\hat{f},
$$ 
$$
\widehat{P_{lo} f} = \phi \hat{f},\quad \widehat{P_{hi} f} = (1-\phi) \hat{f}.
$$
We write $P_{\pm lo}=P_{\pm}P_{lo}$ and $P_{\pm hi}=P_{\pm}P_{hi}$ and let $\chi_\cdot$ be the Fourier multiplier corresponding to $P_\cdot$. These projections are well-defined whenever $\hat{f}\in L^1_{loc}$. We define the gauge transform
$$
V=e^{-iF/2}-1,\quad w=\partial_{x}V=ue^{-iF/2},
$$
and thus
\begin{equation}\label{eq:ucomow}
u=we^{iF/2}=(1+\bar{V})\partial_xV.
\end{equation}
Observe that
$$
\|V(t)\|_{H^{s+1}}\le \sum_{n\ge 1} \frac{\|F(t)\|_{H^{s+1}}^n}{2^nn!}\lesssim \|F(t)\|_{H^{s+1}}e^{\|F(t)\|_{H^{s+1}}/2}.
$$
Set
$$
V_\pm = P_{\pm hi}V,\quad V_{lo}=P_{lo}V, \quad w_\pm=P_{\pm hi}w, w_{lo}=P_{lo}w.
$$
It follows from \eqref{eq:eqF} that
\begin{align}
(V_+)_t + H(V_+)_{xx}& = (V_+)_t - i(V_+)_{xx} \\&= P_{+hi}\left[ -\frac{i}{2} e^{-iF/2}\left( F_t - i F_{xx} -  \frac{1}{2}( F_x)^2 \right) \right] \\&= P_{+hi}\left[ - e^{-iF/2}\left( P_- F_{xx} \right) \right]\label{eq:eqV}\\&= P_{+hi}\left[ - (e^{-iF/2}-1)\left( P_- F_{xx} \right) \right]
\\&= -P_{+hi}\left(V_+P_-\partial_{x}((1+\bar{V})\partial_xV)\right)=:Q_+(V)+C_+(V)
\end{align}
where $Q$ is quadratic and $C$ is cubic.
Analogously, the equation in the large negative frequency range is
\begin{equation}\label{eq:eqV2}
(V_-)_t + H(V_-)_{xx}
= -P_{-hi}\left(V_-P_+\partial_{x}((1+\bar{V})\partial_xV)\right)=:Q_-(V)+C_-(V)
\end{equation}
In the low spectrum range, we have
$$
(V_{lo})_t + H(V_{lo})_{xx} = -P_{+lo}\left(V P_-\partial_{x}u\right) -P_{-lo}\left(V P_+\partial_{x}u\right)
$$
Due to the restriction on the frequencies, the two terms are equivalent to $P_{lo}(wu)$. Since $w,u \in L^\infty((0,T), L^2)$, $\widehat{wu}\in L^\infty((0,T)\times\real)$ and $(V_{lo})_t\in L^\infty((0,T), H^\infty(\real))$.

\begin{lem}\label{lem:derivadaslimitadas}
	For $s\ge 0$, one has
	$$
\| (\tilde{V}_\pm)_t\|_{L^\infty((0,T)\times \real)}\lesssim \|V\|_{L^\infty((0,T),H^1)}^2 + \|V\|_{L^\infty((0,T),H^1)}^3.
	$$
\end{lem}
\begin{proof}
It suffices to prove that $\widehat{C}_\pm, \widehat{Q}_\pm \in L^\infty((0,T)\times \real)$. For $|\xi|<1$, $C\equiv Q \equiv 0$. For $|\xi|>1$, let us write $\widehat{Q}_+$:
$$
\widehat{Q}_+(V)=\int_{\xi_1>\xi>1} (\xi-\xi_1)\hat{V}(\xi_1)\hat{w}(\xi-\xi_1)d\xi_1 = \int_{\xi_1>\xi>1} \frac{\xi-\xi_1}{\xi_1}\left(\xi_1\hat{V}(\xi_1)\right)\hat{w}(\xi-\xi_1)d\xi_1 .
$$
Due to the restriction in the frequency space, the multiplier $(\xi-\xi_1)/\xi_1$ is bounded and
$$
\|\widehat{Q}_+(V)\|_{L^\infty((0,T)\times\real) }\lesssim \int |\xi_1\hat{V}(\xi_1)\hat{w}(\xi-\xi_1)|d\xi_1\lesssim \|V\|_{H^1}^2.
$$
The same argument applies for $Q_-$ and $C_\pm$.
\end{proof}

\section{The infinite normal form reduction}

In this section, we briefly explain the INFR method for general dispersive equations and how it can be used to prove nonlinear smoothing and uniqueness. We then conclude by stating the necessary adaptations to the Benjamin-Ono case. The discussion follows closely \cite{CorreiaSilva} and \cite{koy}.

For the sake of simplicity, let us start with a nonlinear dispersive equation with a single polynomial nonlinearity of order $k$,
$$
u_t + iL(D)u = N(u).
$$
Written in terms of the profile $\tilde{u}(t)=e^{-itL(D)}u(t)$,
\begin{align}
\tilde{u}(t,\xi) &= \tilde{u}(0,\xi) + \int_0^t \int_{\xi_1+\dots+\xi_k=\xi} e^{is\Phi(\Xi)}m(\Xi)\tilde{u}(s,\xi_1)\dots \tilde{u}(s,\xi_k)d\Xi ds\nonumber\\&=: \tilde{u}_0 + \int_0^t \mathcal{N}^{(1)}(\tilde{u}(s))ds\label{eq:geral}
\end{align}
We abuse the notation a bit, by forgetting possible complex conjugates of $\tilde{u}$. 

We introduce some operators which will be crucial for the subsequent analysis. Given $\sigma\in (0,1)$, define the \textit{phase-weighted operator} as
\begin{equation}
\mathcal{F}\left[T_\sigma(u_1,\dots,u_k)\right](\x)=\int_{\x_1+\dots + \x_k=\x} \frac{1}{\jap{\Phi(\Xi)}^\sigma}m(\Xi)\hat{u}_1(\x_1)\dots \hat{u}_k(\x_k) d\x_1\dots d\x_{k-1}.
\end{equation}
Moreover, for any $\alpha\in \real$ and $M>0$, define the \textit{frequency-restricted operator}
$$
\mathcal{F}[T^{\alpha,M}(u_1,\dots, u_k)](\xi)=\int_{\substack{\x_1+\dots + \x_k=\x\\|\Phi(\Xi)-\alpha|<M}}m(\Xi)\hat{u}_1(\xi_1)\dots \hat{u}_k(\xi_k)d\x_1\dots d\x_k,
$$
The main assumption is the existence of appropriate multilinear bounds for these operators.

\begin{assumpt}
	For $\epsilon>0$ fixed, there exists $\sigma\in (0,1)$ such that
	\begin{equation}\label{eq:basicaepsilon}\tag{Bound$_{\sigma,\epsilon}$}
	\|T_\sigma(u_1,\dots,u_k)\|_{H^{s+\epsilon}}\lesssim \prod_{j=1}^{k}\|u_j\|_{H^s}.
	\end{equation}
	Moreover, there exist $\gamma,\beta$ satisfying
	$$
	\beta>0,\quad  \theta:=1-\max\{\gamma+\beta, \sigma+\gamma\}>0.
	$$
	such that, for any $\alpha\in \real$ and $M>1$,
	\begin{equation}\label{eq:frequencyrestrictedgen}\tag{Bound$^{\alpha,M}$}
	\|T^{\alpha,M}(u_1,\dots,u_k)\|_{H^s}\lesssim  \sup\{\jap{\alpha}^{\gamma},M^{\gamma}\}M^{\beta}\prod_{j=1}^{k}\|u_k\|_{H^s}
	\end{equation}
\end{assumpt}

\begin{nb}\label{nota:alphaimplicasigma}
	Observe that we only require a gain of regularity for the phase-weighted operator, which is centered around $\Phi=0$. For the frequency-restricted one, the phase is centered around $\alpha$, which may create additional difficulties in achieving a higher regularity. This separation is often useful (see \cite{CorreiaSilva}). However, it may happen that even for the frequency-restricted operator, a gain is already possible:
	\begin{equation}\label{eq:frequencyrestrictedepsilon}\tag{Bound$^{\alpha,M}_\epsilon$}
	\|T^{\alpha,M}(u_1,\dots,u_k)\|_{H^{s+\epsilon}}\lesssim  \sup\{\jap{\alpha}^{\gamma},M^{\gamma}\}M^{\beta}\prod_{j=1}^{k}\|u_k\|_{H^s}
	\end{equation}
	In this case, one has freely (Bound$_{\sigma,\epsilon}$) for $\sigma>\gamma+\beta$. Indeed, decomposing dyadically in $\Phi$,
\begin{align*}
	\|T_\sigma(u_1,\dots,u_k)\|_{H^{s+\epsilon}}&\lesssim \sum_{M \mbox{ dyadic}} \frac{1}{M^{\sigma}}\|T^{0,M}(u_1,\dots,u_k)\|_{H^{s+\epsilon}} \\&\lesssim \sum_{M \mbox{ dyadic}} M^{\gamma+\beta-\sigma} \prod_{j=1}^{k}\|u_k\|_{H^s} \lesssim \prod_{j=1}^{k}\|u_k\|_{H^s}.
\end{align*}
If $2\gamma+\beta<1$, the conditions of Assumption 1 are directly verified.
\end{nb}

For a fixed $N>1$ (which will be determined later), let us split the frequency domain into the near-resonant and nonresonant regions, depending on whether $|\Phi|$ is smaller or greater than $N$. We use the subscripts $1$ for the near-resonant term and $2$ for the nonresonant one:
$$
\tilde{u}(t,\x)=\tilde{u}(0,\x) + \int_0^t \mathcal{N}_1^{(1)}(\tilde{u}(s)) +\mathcal{N}_2^{(1)}(\tilde{u}(s)) ds.
$$
Using \eqref{eq:basicaepsilon}, 

\begin{equation}\label{eq:primeirotermo}
\left\|\mathcal{F}^{-1}\left(\mathcal{N}_1^{(1)}(\tilde{u}(t))\right)\right\|_{H^{s+\epsilon}}\lesssim N^\sigma\|T_\sigma(\tilde{u},\dots,\tilde{u})(t)\|_{H^{s+\epsilon}}\lesssim N^\sigma\|u(t)\|_{H^s}^k.
\end{equation}
For $\mathcal{N}_2^{(1)}$, one integrates by parts in time, using the relation
$$
e^{is\Phi}=\partial_s\left(\frac{1}{i\Phi}e^{is\Phi}\right)
$$
so that
$$
\int_0^t \mathcal{N}_2^{(1)}(\tilde{u}(s)) ds = \left[\mathcal{N}_0^{(2)}(\tilde{u}(s)) \right]_{s=0}^{s=t} + \int_0^t \mathcal{N}^{(2)}(\tilde{u}(s)) ds.
$$
Due to the factor $1/\Phi$ and the restriction on the frequency domain, the boundary terms are easily bounded:
$$
\left\|\mathcal{F}^{-1}\left(\mathcal{N}_0^{(2)}(\tilde{u}(s))\right)\right\|_{H^{s+\epsilon}}\lesssim \frac{1}{N^{1-\sigma}}\|T_\sigma(\tilde{u},\dots,\tilde{u})(s)\|_{H^{s+\epsilon}}\lesssim N^{-1+\sigma}\|u\|_{H^s}.
$$
For the remainder $\mathcal{N}^{(2)}$, one uses \eqref{eq:geral} to replace $\tilde{u}_t$ and obtain a new oscillatory integral which is of order $2(k-1)+1$ in $\tilde{u}$. 

So far, we have rewritten \eqref{eq:geral} as
$$
\tilde{u}(t,\xi) = \tilde{u}(0,\xi) + \mbox{controllable terms} + \int_0^t \mathcal{N}^{(2)}(\tilde{u}(s))ds.
$$

This concludes the first step in the INFR. One may now apply a recursive algorithm to expand the remainder integral $\mathcal{N}^{(2)}$. Fix $0<\delta<\theta/\beta$ and the sequence
$$
c_j=(j+1)^{2/\theta},\quad j\in \mathbb{N}.
$$
\vskip10pt
\underline{At the $J$-th step,}
\vskip10pt
\noindent\textit{Step 1.} Split the frequency domain into
\begin{itemize}
	\item Near-resonant: $\mbox{|Phase|}<c_J\mbox{|Phase at step 1|}^{\delta}$;
	\item Nonresonant: $\mbox{|Phase|}>c_J\mbox{|Phase at step 1|}^{\delta}$;
\end{itemize}
and write $$\int_0^t\mathcal{N}^{(J+1)}(\tilde{u}(s))ds=\int_0^t\mathcal{N}_1^{(J+1)}(\tilde{u}(s))+\mathcal{N}_2^{(J+1)}(\tilde{u}(s))ds.
$$
\textit{Step 2.} Integrate by parts in time the nonresonant term:
$$
\int_0^t \mathcal{N}_2^{(J+1)}(\tilde{u}(s))ds = \left[\mathcal{N}_0^{(J+2)}(\tilde{u}(s))  \right]_{s=0}^{s=t} + \int_0^t \mathcal{R}^{(J+2)}(\tilde{u}(s))ds.
$$
\textit{Step 3.} Use \eqref{eq:geral} to replace all instances of $\tilde{u}_t$ in $\mathcal{R}$ in order to obtain an oscillatory integral $\mathcal{N}^{(J+2)}$ of order $(k-1)(J+2)-1$ in $\tilde{u}$.
\vskip5pt
\noindent\textit{Step 4.} Repeat the algorithm for $\mathcal{N}^{(J+2)}(\tilde{u})$.

\vskip10pt
After an infinite amount of steps, one formally obtains the normal form equation
\begin{equation}\label{eq:NFE}\tag{NFE}
\tilde{u}(t,\xi)=\tilde{u}(0,\xi) + \sum_{j\ge 2} \left[\mathcal{N}_0^{(j)}(\tilde{u}(s))\right]_{s=0}^{s=t} + \int_0^t \sum_{j\ge 1} \mathcal{N}_1^{(j)}(\tilde{u}(s))ds.
\end{equation}


The basic multilinear estimates \eqref{eq:basicaepsilon} and \eqref{eq:frequencyrestrictedgen} can be propagated throughout the algorithmic procedure.

\begin{lem}[$H^s$ bounds for the \eqref{eq:NFE}]\label{lem:HsboundsNFE}
	For any $J\ge 2$,
	$$
	\left\|\mathcal{F}^{-1}[\mathcal{N}_1^{(J)}(\tilde{u})]\right\|_{H^{s+\epsilon}} \lesssim N^{-\theta-\delta\theta(J-2)+\delta\beta}\|u\|_{H^s}^{J(k-1)+1},
	$$
	$$
	\left\|\mathcal{F}^{-1}[\mathcal{N}_0^{(J+1)}(\tilde{u})]\right\|_{H^{s+\epsilon}}\lesssim N^{-\theta-\delta\theta(J-2)+\delta(\beta-1)}\|u\|_{H^s}^{J(k-1)+1},
	$$
	$$
	\left\|\mathcal{F}^{-1}[\mathcal{N}_1^{(J)}(\tilde{u}) - \mathcal{N}_1^{(J)}(\tilde{v})]\right\|_{H^{s+\epsilon}} \lesssim N^{-\theta-\delta\theta(J-2)+\delta\beta}\left(\|u\|_{H^s}^{J(k-1)} + \|v\|_{H^s}^{J(k-1)}\right)\|u-v\|_{H^s}
	$$
	and
	$$
	\left\|\mathcal{F}^{-1}[\mathcal{N}_0^{(J+1)}(\tilde{u})- \mathcal{N}_0^{(J+1)}(\tilde{v})]\right\|_{H^{s+\epsilon}}\lesssim N^{-\theta-\delta\theta(J-2)+\delta(\beta-1)}\left(\|u\|_{H^s}^{J(k-1)} + \|v\|_{H^s}^{J(k-1)}\right)\|u-v\|_{H^s}.
	$$
\end{lem}

\begin{nb}
	Since the estimates of Lemma \ref{lem:HsboundsNFE} decay exponentially in $J$ for $N\gg M:=\|u\|_{L^\infty((0,T),H^s)}$, one obtains a formal bound on \eqref{eq:NFE},
	\begin{align*}
	\|\mathcal{F}^{-1}[\tilde{u}(t)-\tilde{u}(0)]\|_{H^{s+\epsilon}} &\lesssim \sum_{J\ge 2} \left\|\mathcal{F}^{-1}[\mathcal{N}_0^{(J)}(\tilde{u}(t)) ]\right\|_{H^{s+\epsilon}} + \sum_{J\ge 2} \left\|\mathcal{F}^{-1}[\mathcal{N}_0^{(J)}(\tilde{u}(0)) ]\right\|_{H^{s+\epsilon}}\\&\qquad + \sum_{J\ge 1} \int_0^t \left\|\mathcal{F}^{-1}[\mathcal{N}_1^{(J)}(\tilde{u}(s)) ]\right\|_{H^{s+\epsilon}}ds\\&\lesssim \sum_{J\ge 2} N^{-\theta-\delta\theta(J-2)+\delta\beta}M^{(k-1)J+1} + TN^{1-\sigma}\\&\qquad +T\sum_{J\ge 2}N^{-\theta-\delta\theta(J-2)+\delta(\beta-1)}M^{(k-1)(J+1)+1}\\&\lesssim C\left(T,\|u\|_{L^\infty((0,T),H^s)}\right),
	\end{align*}
	which is precisely the nonlinear smoothing property. This can be shown to be valid for smooth initial data (such that $\tilde{u}_t\in H^s$). Therefore, if a nice well-posedness theory is available, one may argue by density and obtain the nonlinear smoothing result without any further considerations (this was the approach taken in \cite{CorreiaSilva}). However, for the Benjamin-Ono equation, due to the lack of a proper well-posedness result at low regularity, this argument only proves that solutions which can be approximated by smooth ones satisfy the nonlinear smoothing property.
\end{nb}

To make use of the formal computations, we have to justify the INFR procedure, which uses three properties:
\begin{enumerate}[label=(\subscript{P}{{\arabic*}})]
	\item  The product rule for the time derivative can be applied for a.e. $\xi\in \real$;
	\item  One can switch time derivatives and integrals in $\xi$;
	\item The remainder tends to zero as $J\to \infty$ (in some weaker norm).
\end{enumerate}
The first property can be ensured if, for a.e. $\xi\in \real$, $\tilde{u}_t(\cdot,\xi)\in L^\infty(0,T)$. As a rule of thumb, this can  be seen if $k<2d/(d-2s)^+$: since $u^k\in C((0,T),L^1(\real^d))$, by Riemann-Lebesgue Lemma, $\mathcal{F}(|u|^k)\in C((0,T)\times \real^d)$. Thus, for $\xi$ fixed, $\mathcal{F}(|u|^k)\in C((0,T))$. Since the difference between $|u|^k$ and $N(u)$ are absolute values and possibly spatial derivatives, this should be enough to see that  $\mathcal{F}(N(u))(\cdot,\xi)\in C((0,T))$. 

The second property can be justified at the $H^s$ level, for $s\ge 0$, following \cite[Remark 4.5]{koy}. To prove that the remainder tends to $0$, we make the following
\begin{assumpt}
	Take $\gamma, \beta$ and $\sigma$ as in Assumption 1. There exists $\mu\in \real$ such that 
	\begin{equation}\label{eq:definormamu}
	\| \jap{\cdot}^{-\mu} \tilde{u}_t\|_{L^\infty}<\infty.
	\end{equation}
	Moreover, one has the following weak phase-weighted and frequency-restricted estimates:
	\begin{equation}\label{eq:basicaepsilonfraca}
	\| \jap{\cdot}^{-\mu} T_\sigma(u_1,\dots,u_k)\|_{L^\infty}\lesssim \min_j\left\{ \|\jap{\cdot}^{-\mu} u_j \|_{L^\infty}  \prod_{l\neq j}\|u_l\|_{H^s}\right\}
	\end{equation}
	\begin{equation}\label{eq:freqrestrfraca}
	\| \jap{\cdot}^{-\mu} T^{\alpha,M}(u_1,\dots,u_k)\|_{L^\infty}\lesssim (|\alpha|+M)^{\gamma}M^\beta \min_j\left\{ \|u_j\|_{L^\infty}  \prod_{l\neq j}\|u_l\|_{H^s}\right\}.
	\end{equation}
\end{assumpt}

Notice that, at each step of the INFR, the remainder term $\mathcal{R}$ is essentially the boundary term $\mathcal{N}_0$ with one $\tilde{u}$ replaced by $\tilde{u}_t$. Then, applying the weak phase-weighted and frequency-restricted estimates, we have
\begin{lem}
	For any $J\ge 2$,
	$$
	\| \jap{\cdot}^{-\mu} \mathcal{R}^{(J+2)}(\tilde{u})\|_{L^\infty}\lesssim N^{-\theta-\delta\theta(J-2)+\delta(\beta-1)}\|u\|_{H^s}^{J(k-1)}\| \jap{\cdot}^{-\mu}\tilde{u}_t\|_{L^\infty}.
	$$
\end{lem}

\begin{thm}\label{teo:geral}
	Fix $s\ge 0$.
	\begin{enumerate}
		\item Under Assumptions 1 and 2 (with $\epsilon=0$), if ($P_1$) holds, then equation \eqref{eq:geral} has a unique solution. Moreover, the data-to-solution map is Lipschitz continuous in $H^s$. 
		\item Given $\epsilon>0$, under Assumptions 1 and 2, if ($P_1$) holds, then equation \eqref{eq:geral} satisfies the nonlinear smoothing property of order $\epsilon$.
	\end{enumerate}
\end{thm}
\begin{proof}
	We start with the uniqueness statement. Given two solutions $u,v \in C([0,T], H^s(\real^d))$, the hypothesis ensure that both solutions satisfy \eqref{eq:NFE}. Applying the Lipschitz estimates of Lemma \ref{lem:HsboundsNFE}, for $N\gg \|u\|_{L^\infty((0,T), H^s)} + \|v\|_{L^\infty((0,T), H^s)}$, one has
	$$
	\|u-v\|_{L^\infty((0,T), H^s)}\lesssim \|u_0-v_0\|_{H^s}
	$$
	and the first claim follows. Given $\epsilon>0$, by hypothesis, the formal bounds from Lemma \ref{lem:HsboundsNFE} can be rigorously justified. The nonlinear smoothing property now follows from Remark 4.
\end{proof}

\begin{nb}
	All of the above considerations remain valid for systems of several dispersive equations and several nonlinear terms: one requires phase-weighted and frequency-restricted estimates for each term separately, for a uniform choice of parameters $\gamma, \beta$ and $\sigma$. 
\end{nb}

\begin{nb}
The equations for  $V_\pm$ involve the low-frequency part  $V_{lo}$. As time derivatives fall onto this factor, \textit{we do not replace it by its evolution equation}. Instead, we use the fact that it is bounded in $H^s$, which is enough for the application of the iterated estimates. For these terms, no further reduction is necessary. A similar strategy was taken in \cite{CorreiaSilva} for the derivative nonlinear Schrödinger equation.
\end{nb}

\section{Proof of the main results}

	In the context of the Benjamin-Ono equation, property (\subscript{P}{1}) follows directly from Lemma \ref{lem:derivadaslimitadas}. Therefore, by Theorem \ref{teo:geral}, our work is reduced to the proof of bounds \eqref{eq:frequencyrestrictedgen} and \eqref{eq:basicaepsilon} (in both strong and weak forms). Following Remark \ref{nota:alphaimplicasigma}, we further reduce our problem by showing the stronger version \eqref{eq:frequencyrestrictedepsilon} for each nonlinear term.
\begin{lem}[Frequency-restricted estimate for the cubic nonlinearities]\label{lem:cubic}
For $s>0$ and $0\le\epsilon<\min\{s,3/4\}$, set $$\gamma(\epsilon)=\max\left\{\frac{1}{4}+\frac{\epsilon-s}{2}, \epsilon-\frac{1}{2}, 0\right\}.
$$
One has
\begin{equation}\label{eq:estimcubicforte}
	\left\|T^{\alpha,M}[C_\pm](V_1,V_2,V_3) \right\|_{H^{s+\epsilon+1}}\lesssim (|\alpha|+ M)^{\gamma(\epsilon)}M^{1/2}\|V_1\|_{H^{s+1}}\|V_2\|_{H^{s+1}}\|V_3\|_{H^{s+1}}
\end{equation}
\begin{align}
&\left\|T^{\alpha,M}[C_\pm](V_1,V_2,V_3) \right\|_{L^\infty}\lesssim (|\alpha|+ M)^{\gamma(0)}M^{1/2}\label{eq:estimcubicfraco} \min_j\left\{ \|\hat{V}_j\|_{L^\infty}\prod_{k\neq j}\|V_k\|_{H^{s+1}}\right\}
\end{align}
\end{lem}
\begin{proof}
	We begin with \eqref{eq:estimcubicforte}. By duality and Cauchy-Schwarz, the estimate is reduced to a bound on
	$$
	I^{\alpha,M}_\epsilon=\int_{\xi>\xi_1>1} \frac{\jap{\xi}^{2s+2\epsilon+2}|\xi_2+\xi_3|^2   }{ \xi_1^2\jap{\xi_1}^{2s}\jap{\xi_2}^{2s+2} \jap{\xi_3}^{2s}}\mathbbm{1}_{|\Phi-\alpha|<M}d\xi_1d\xi_2 
	$$
	where
	$$
	\Phi=|\xi|\xi - |\xi_1|\xi_1 + |\xi_2|\xi_2 -|\xi_3|\xi_3.
	$$
	Since $\jap{\xi}\le \jap{\xi_1}\jap{\xi_2}\jap{\xi_3}$, we may suppose that $s$ is as close to $\epsilon$ as we desire.
	
	Region A: $|\xi_1|\ll |\xi_2|, |\xi_3|$. Then $|\xi_2|\sim |\xi_3|\gg |\xi|$ and, writing $\xi_1=\xi-\xi_2-\xi_3$,
	$$
	|\partial_{\xi_2}\Phi| = 2||\xi_1|+|\xi_2||\sim |\xi_2|.
	$$
	Hence
	\begin{align*}
	I^{\alpha,M}_\epsilon \lesssim \int \frac{|\xi_2|^{2+2\epsilon}}{\jap{\xi_2}\jap{\xi_3}^{4s+1}}\mathbbm{1}_{|\Phi-\alpha|<M} \frac{d\Phi}{|\xi_2|}d\xi_3 \lesssim M.
	\end{align*}
	
	Region B: $|\xi_2|\lesssim |\xi_1|\ll |\xi_3|$ or $|\xi_3|\lesssim |\xi_1|\ll |\xi_2|$. This implies that $|\xi|\gg |\xi_1|$, which is impossible.
	
	Region C: $|\xi_1|\gtrsim |\xi_2|, |\xi_3|$.  Then, writing $\xi_2=\xi-\xi_1-\xi_3$,
	$$
	|\partial_{\xi_1}\Phi|\gtrsim |\xi_1|.
	$$
	Case I. $|\xi_2|\gtrsim |\xi_3|$. If $|\xi_2|\gtrsim |\xi|$,
	$$
	I^{\alpha,M}_\epsilon\lesssim \int_{\xi>\xi_1>1} \frac{\jap{\xi}^{2s+2\epsilon+2}|\xi_2|^2   }{ \xi_1^2\jap{\xi_1}^{2s}\jap{\xi_2}^{2s+2} \jap{\xi_3}^{2s}}\mathbbm{1}_{|\Phi-\alpha|<M}\frac{d\Phi}{|\xi_1|}d\xi_3 \lesssim \int_{\xi>\xi_1>1}\frac{1}{\jap{\xi_3}^{2s+1}} \mathbbm{1}_{|\Phi-\alpha|<M}d\Phi d\xi_3\lesssim M
	$$
	If $|\xi_2|\ll |\xi|$, since $|\partial_{\xi_1}\Psi|\gtrsim |\xi_1|\sim |\xi|$,
	$$
	|\Psi|\gtrsim |\xi(\xi_2+\xi_3)|.
	$$
Hence, for $a=\max\{2\epsilon-1,0\}$ and since $\epsilon<\min\{3/4, s\}$,
\begin{align*}
	I^{\alpha,M}_\epsilon&\lesssim \int \frac{|\xi|^{2\epsilon}|\xi_2+\xi_3|^2}{\jap{\xi_2}^{2s+1}\jap{\xi_3}^{2s+1}}\mathbbm{1}_{|\Phi-\alpha|<M} \frac{d\Phi}{|\xi_1|}d\xi_3\\&\lesssim  \int \frac{|\xi|^{2\epsilon-a}|\xi_2+\xi_3|^{2-a}|\Phi|^a}{\jap{\xi_2}^{2s+1}\jap{\xi_3}^{2s+1}}\mathbbm{1}_{|\Phi-\alpha|<M} \frac{d\Phi}{|\xi_1|}d\xi_3\\&\lesssim (|\alpha|+M)^a\int \frac{1}{\jap{\xi_3}^{2s+1}}\mathbbm{1}_{|\Phi-\alpha|<M} d\Phi d\xi_3 \lesssim (|\alpha|+M)^aM,
\end{align*}
	Case II. $|\xi_2|\ll |\xi_3|$. Then $\xi_3<0$. If $|\xi|\lesssim |\xi_2|$, we proceed as in Case I.
	Otherwise, since
	$$
	\Phi = \xi^2 -(\xi-\xi_2-\xi_3)^2 \pm \xi_2^2 + \xi_3^2 \sim \xi\xi_3,
	$$
	we have $|\xi\xi_3|^{1/2}\lesssim (|\alpha|+M)^{1/2}$. Writing $\xi_3=\xi-\xi_1-\xi_2$,
	$$
	|\partial_{\xi_1}\Phi|\sim ||\xi_1|-|\xi_3||=|\xi_1+\xi_3|\sim |\xi|.
	$$
	Therefore, since $2\epsilon-1<2-2s$,
	\begin{align*}
	I^{\alpha,M}_\epsilon&\lesssim \int \frac{|\xi|^{2s+2\epsilon+2}|\xi_3|^2}{|\xi_1|^{2+2s}\jap{\xi_3}^{2s} \jap{\xi_2}^{2s+2} } \mathbbm{1}_{|\Phi-\alpha|<M}d\xi_1 d\xi_2 \\&\lesssim \int \frac{|\xi|^{2s+2\epsilon+1}|\xi_3|^2}{|\xi_1|^{2+2s}|\xi_3|^{2s} \jap{\xi_2}^{2s+2} } \mathbbm{1}_{|\Phi-\alpha|<M}d\Phi d\xi_2  \\&\lesssim \int \frac{|\xi|^{2\epsilon-1}|\xi_3|^{2-2s}}{\jap{\xi_2}^{2s+2} } \mathbbm{1}_{|\Phi-\alpha|<M}d\Phi d\xi_2  \\&\lesssim \int \frac{|\xi\xi_3|^{1/2+\epsilon-s}}{ \jap{\xi_2}^{2s+2} } \mathbbm{1}_{|\Phi-\alpha|<M}d\Phi d\xi_2 \lesssim (|\alpha|+M)^{1/2+\epsilon-s}M.
	\end{align*}
	For estimate \eqref{eq:estimcubicfraco}, we bound directly
\begin{align*}
	&\|T^{\alpha,M}[C_+](V_1,V_2,V_3)\|_{L^\infty}\lesssim \int_{\xi_1>\xi>1} \frac{|\xi_2+\xi_3|}{\xi_1}\mathbbm{1}_{|\Phi-\alpha|<M}\hat{w}_1(\xi_1)\hat{V}_2(\xi_2)\hat{w}_3(\xi_3)d\xi_1d\xi_2 \\&\lesssim \left(\int_{\xi_1>\xi>1}\frac{|\xi_2+\xi_3|^2\jap{\xi_1}^2}{|\xi_1|^2\jap{\xi_2}^{2s+2}\jap{\xi_3}^{2s} }\mathbbm{1}_{|\Phi-\alpha|<M} d\xi_1d\xi_2\right)^{1/2}\|V_1\|_{L^\infty}\|V_2\|_{H^{s+1}}\|V_3\|_{H^{s+1}}
\end{align*}
As the integral is quite similar to $I^{\alpha,M}_0$, the proof for \eqref{eq:estimcubicforte} easily extends to this case. The same argument can be applied for $C_-$ and for the other combinations of $L^\infty$ and $H^s$ norms.
\end{proof}
\begin{lem}[Frequency-restricted estimate for the quadratic nonlinearities]\label{lem:quad}
	For $s\ge 0$ and $0\le\epsilon<\min \{s,3/4\}$, set
	$$
	\gamma(\epsilon)=\max\left\{\frac{1}{2}+\epsilon-s, 0\right\}.
	$$
	Then
	\begin{equation}\label{eq:estimquadforte}
	\left\|T^{\alpha,M}[Q_\pm](V_1,V_2) \right\|_{H^{s+\epsilon+1}}\lesssim (|\alpha|+ M)^{\gamma(\epsilon)}M^{1/2}\|V_1\|_{H^{s+1}}\|V_2\|_{H^{s+1}}
	\end{equation}
	\begin{align}
	&\left\|T^{\alpha,M}[Q_\pm](V_1,V_2) \right\|_{L^\infty}\lesssim (|\alpha|+ M)^{\gamma(0)}M^{1/2}\label{eq:estimquadfraco} \min\left\{ \|\hat{V}_1\|_{L^\infty}\|V_2\|_{H^{s+1}}, \|V_1\|_{H^{s+1}}\|\hat{V}_2\|_{L^\infty} \right\}
	\end{align}
\end{lem}
\begin{proof}
Here, we are led to study the integral
$$
J^{\alpha,M}=\int_{\xi>\xi_1>1} \frac{\jap{\xi}^{2s+2\epsilon+2}\xi_2^2}{\xi_1^2\jap{\xi_1}^{2s}\jap{\xi_2}^{2s}} \mathbbm{1}_{|\Psi-\alpha|<M}d\xi_2,
$$
where $\Psi=\xi^2-\xi_1^2 +\xi_2^2=2\xi\xi_2$. Once again, we assume that $s$ is close to $\epsilon$. Without loss of generality, suppose that $|\xi|>|\xi_2|$. Since $|\partial_{\xi_2}\Psi|\sim |\xi|$,
\begin{align*}
J^{\alpha,M}&\lesssim \int_{\xi>\xi_1>1} \frac{\jap{\xi}^{2s+2\epsilon+1}\xi_2^2}{\xi_1^2\jap{\xi_1}^{2s}\jap{\xi_2}^{2s}} \mathbbm{1}_{|\Psi-\alpha|<M}d\Psi\\&\lesssim \int_{\xi>\xi_1>1} |\xi|^{2\epsilon-1}|\xi_2|^{2-2s} \mathbbm{1}_{|\Psi-\alpha|<M}d\Psi\\& \lesssim \int_{\xi>\xi_1>1} |\Psi|^{\frac{1}{2}+\epsilon-s} \mathbbm{1}_{|\Psi-\alpha|<M}d\Psi \lesssim (|\alpha|+M)^{\frac{1}{2}+\epsilon-s}M,
\end{align*}
where we use, in the penultimate estimate, the fact that $2\epsilon-1<2-2s$. The $L^\infty$ bound follows as for the cubic nonlinearity.
\end{proof}

\begin{proof}[Proof of Theorem \ref{teo:principal}]
The existence of solution and the definition of the gauge mapping are a consequence of the considerations of Section 2. Applying the INFR procedure and the estimates from Lemmas \ref{lem:cubic} and \ref{lem:quad}, the Lipschitz continuity, the nonlinear smoothing property and the uniqueness for the gauge transform $V=\mathcal{G}(u)$ follows. Finally, the uniqueness for $u$ follows from that of $V$.	
\end{proof}

\section{Acknowledgements}
The author would like to thank Jorge Silva and Felipe Linares for helpful suggestions and comments. The author was partially supported by Funda\c{c}\~ao para a Ci\^encia e Tecnologia, through the grant UID/MAT/04459/2019.

\bibliography{biblio}

\begin{thebibliography}{10}

\bibitem{babin}
A.~Babin, A.~Ilyin, and E.~Titi.
\newblock On the regularization mechanism for the periodic {K}orteweg–de
  {V}ries equation.
\newblock {\em Communications on Pure and Applied Mathematics}, 64(5):591--648,
  2011.

\bibitem{benjamin}
T.~Brooke Benjamin.
\newblock Internal waves of permanent form in fluids of great depth.
\newblock {\em Journal of Fluid Mechanics}, 29(3):559–592, 1967.

\bibitem{BS1}
J.L. Bona and J.-C. Saut.
\newblock Dispersive blow up of solutions of generalized {K}d{V} equations.
\newblock {\em J. Differ. Equ.}, 103:3--57, 1993.

\bibitem{bourgain}
J.~Bourgain.
\newblock {\em Global solutions of nonlinear {S}chr\"{o}dinger equations},
  volume~46 of {\em American Mathematical Society Colloquium Publications}.
\newblock American Mathematical Society, Providence, RI, 1999.

\bibitem{burqplanchon}
Nicolas Burq and Fabrice Planchon.
\newblock On well-posedness for the {B}enjamin-{O}no equation.
\newblock {\em Math. Ann.}, 340(3):497--542, 2008.

\bibitem{CorreiaSilva}
Sim\~{a}o Correia and Jorge~Drumond Silva.
\newblock Nonlinear smoothing for dispersive {PDE}: {A} unified approach.
\newblock {\em J. Differential Equations}, 269(5):4253--4285, 2020.

\bibitem{tzirakis6}
M.~B. Erdo\u{g}an, T.~B. G\"{u}rel, and N.~Tzirakis.
\newblock The derivative nonlinear {S}chr\"{o}dinger equation on the half line.
\newblock {\em Ann. Inst. H. Poincar\'{e} Anal. Non Lin\'{e}aire},
  35(7):1947--1973, 2018.

\bibitem{tzirakis3}
M.~B. Erdo\u{g}an and N.~Tzirakis.
\newblock Talbot effect for the cubic non-linear {S}chr\"{o}dinger equation on
  the torus.
\newblock {\em Math. Res. Lett.}, 20(6):1081--1090, 2013.

\bibitem{fonsecalinaresponce}
Germ\'{a}n Fonseca, Felipe Linares, and Gustavo Ponce.
\newblock The {IVP} for the {B}enjamin-{O}no equation in weighted {S}obolev
  spaces {II}.
\newblock {\em J. Funct. Anal.}, 262(5):2031--2049, 2012.

\bibitem{gebalin}
D.A. Geba and B.~Lin.
\newblock Unconditional well-posedness for the {K}awahara equation.
\newblock {\em preprint, arXiv:2007.08923}, 2020.

\bibitem{germain}
Pierre Germain.
\newblock Space-time resonances.
\newblock {\em Journ\'ees \'equations aux d\'eriv\'ees partielles}, 2010.

\bibitem{masmoudi}
Pierre Germain, Nader Masmoudi, and Jalal Shatah.
\newblock Global solutions for 3{D} quadratic {S}chr\"{o}dinger equations.
\newblock {\em Int. Math. Res. Not. IMRN}, (3):414--432, 2009.

\bibitem{guo}
Z.~Guo, S.~Kwon, and T.~Oh.
\newblock {P}oincar\'e-{D}ulac normal form reduction for unconditional
  well-posedness of the periodic cubic {NLS}.
\newblock {\em Commun. Math. Phys.}, 322:19--48, 2013.

\bibitem{HKO}
Nako Hayashi, Keiichi Kato, and Tohru Ozawa.
\newblock Dilation method and smoothing effects of solutions to the
  {B}enjamin-{O}no equation.
\newblock {\em Proc. Roy. Soc. Edinburgh Sect. A}, 126(2):273--285, 1996.

\bibitem{ifrimtataru}
Mihaela Ifrim and Daniel Tataru.
\newblock Well-posedness and dispersive decay of small data solutions for the
  {B}enjamin-{O}no equation.
\newblock {\em Ann. Sci. \'{E}c. Norm. Sup\'{e}r. (4)}, 52(2):297--335, 2019.

\bibitem{ionescukenig}
Alexandru~D. Ionescu and Carlos~E. Kenig.
\newblock Global well-posedness of the {B}enjamin-{O}no equation in
  low-regularity spaces.
\newblock {\em J. Amer. Math. Soc.}, 20(3):753--798, 2007.

\bibitem{iorio}
Rafael~Jos\'{e} I\'{o}rio, Jr.
\newblock On the {C}auchy problem for the {B}enjamin-{O}no equation.
\newblock {\em Comm. Partial Differential Equations}, 11(10):1031--1081, 1986.

\bibitem{imt}
J.~Pedro Isaza, L.~Jorge Mej\'{\i}a, and Nikolay Tzvetkov.
\newblock A smoothing effect and polynomial growth of the {S}obolev norms for
  the {KP}-{II} equation.
\newblock {\em J. Differential Equations}, 220(1):1--17, 2006.

\bibitem{kenigkoenig}
Carlos~E. Kenig and Kenneth~D. Koenig.
\newblock On the local well-posedness of the {B}enjamin-{O}no and modified
  {B}enjamin-{O}no equations.
\newblock {\em Math. Res. Lett.}, 10(5-6):879--895, 2003.

\bibitem{keraani}
S.~Keraani and A.~Vargas.
\newblock A smoothing property for the ${L}^{2}$-critical {NLS} equations and
  an application to blowup theory.
\newblock {\em Annales de l'I.H.P. Analyse non lin\'eaire}, 26(3):745--762,
  2009.

\bibitem{kishimoto1}
N.~Kishimoto.
\newblock Unconditional uniqueness for the periodic modified benjamin-ono
  equation by normal form approach.
\newblock {\em preprint, arXiv:1912.01363}, 2019.

\bibitem{kishimoto2}
Nobu Kishimoto.
\newblock Well-posedness of the {C}auchy problem for the {K}orteweg-de {V}ries
  equation at the critical regularity.
\newblock {\em Differential Integral Equations}, 22(5-6):447--464, 2009.

\bibitem{kochtzevtkov2}
H.~Koch and N.~Tzvetkov.
\newblock On the local well-posedness of the {B}enjamin-{O}no equation in
  {$H^s({\Bbb R})$}.
\newblock {\em Int. Math. Res. Not.}, (26):1449--1464, 2003.

\bibitem{kochtzevtkov1}
H.~Koch and N.~Tzvetkov.
\newblock Nonlinear wave interactions for the {B}enjamin-{O}no equation.
\newblock {\em Int. Math. Res. Not.}, (30):1833--1847, 2005.

\bibitem{ko}
S.~Kwon and T.~Oh.
\newblock On unconditional well-posedness of modified {K}d{V}.
\newblock {\em Int. Math. Res. Not. IMRN}, (15):3509--3534, 2012.

\bibitem{koy}
S.~Kwon, T.~Oh, and H.~Yoon.
\newblock Normal form approach to unconditional well-posedness of nonlinear
  dispersive {PDE}s on the real line.
\newblock {\em Annales de la Faculte des Sciences de Toulouse}, 11 2018.

\bibitem{linaresscialom}
F.~Linares and M.~Scialom.
\newblock On the smoothing properties of solutions to the modified
  {K}orteweg-de {V}ries equation.
\newblock {\em J. Differential Equations}, 106(1):141--154, 1993.

\bibitem{molinetsauttzevtkov}
L.~Molinet, J.~C. Saut, and N.~Tzvetkov.
\newblock Ill-posedness issues for the {B}enjamin-{O}no and related equations.
\newblock {\em SIAM J. Math. Anal.}, 33(4):982--988, 2001.

\bibitem{molinetpilod}
Luc Molinet and Didier Pilod.
\newblock The {C}auchy problem for the {B}enjamin-{O}no equation in {$L^2$}
  revisited.
\newblock {\em Anal. PDE}, 5(2):365--395, 2012.

\bibitem{ono}
Hiroaki Ono.
\newblock Algebraic solitary waves in stratified fluids.
\newblock {\em J. Phys. Soc. Japan}, 39(4):1082--1091, 1975.

\bibitem{ponce}
Gustavo Ponce.
\newblock On the global well-posedness of the {B}enjamin-{O}no equation.
\newblock {\em Differential Integral Equations}, 4(3):527--542, 1991.

\bibitem{tao}
Terence Tao.
\newblock Global well-posedness of the {B}enjamin-{O}no equation in {$H^1({\bf
  R})$}.
\newblock {\em J. Hyperbolic Differ. Equ.}, 1(1):27--49, 2004.

\end{thebibliography}
\bibliographystyle{plain}

\begin{center}
	{\scshape Simão Correia}\\
	{\footnotesize
		Centro de Matemática, Aplicações Fundamentais e Investigação Operacional,\\
		Department of Mathematics,\\
		Instituto Superior T\'ecnico, Universidade de Lisboa\\
		Av. Rovisco Pais, 1049-001 Lisboa, Portugal\\
		simao.f.correia@tecnico.ulisboa.pt
	}
	
\end{center}

\end{document}